\documentclass[final,11pt,3p]{elsarticle}

\usepackage{graphicx}
\usepackage{amssymb}
\usepackage{amsthm}
\usepackage{amsmath}
\usepackage{algorithm}
\usepackage{algorithmic}
\usepackage{tikz}
\usepackage{lineno}
\usetikzlibrary{shapes}

\usepackage{subcaption}
\usepackage{caption}
\usepackage{float}
\usepackage{geometry}
\usepackage{mathrsfs} 
\usepackage{diagbox}
\biboptions{sort&compress}

\newtheorem{theorem}{Theorem}
\newtheorem{lemma}{Lemma}

\newtheorem{remark}{Remark}
\newtheorem{corollary}{Corollary}

\theoremstyle{definition}
\newtheorem{assumption}{Assumption}

\newcommand{\mB}{\mathcal{B}}
\newcommand{\mH}{\mathcal{H}}
\newcommand{\mLV}{\mathcal{LV}}
\newcommand{\mT}{\mathcal{T}}
\newcommand{\mU}{\mathcal{U}}
\newcommand{\mV}{\mathcal{V}}
\newcommand{\mX}{\mathcal{X}}
\newcommand{\mY}{\mathcal{Y}}
\newcommand{\mZ}{\mathcal{Z}}

\usepackage{xcolor}

\begin{document}

\begin{tikzpicture}[remember picture,overlay]
	\node[anchor=north east,inner sep=20pt] at (current page.north east)
	{\includegraphics[scale=0.2]{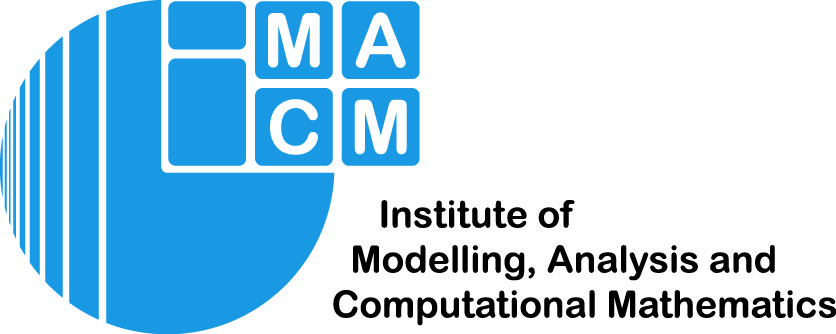}};
\end{tikzpicture}

\begin{frontmatter}

\title{Dynamical Behavior of a Stochastic Epidemiological Model: Stationary Distribution and Extinction of a SIRS Model with Stochastic Perturbations}

\author[AMNEA]{Achraf Zinihi}
\ead{achrafzinihi99@gmail.com}

\author[AMNEA]{Moulay Rchid Sidi Ammi}
\ead{rachidsidiammi@yahoo.fr}

\author[BUW]{Matthias Ehrhardt\corref{Corr}}
\cortext[Corr]{Corresponding author}
\ead{ehrhardt@uni-wuppertal.de}

\address[AMNEA]{Department of Mathematics, MAIS Laboratory, AMNEA Group, Faculty of Sciences and Technics,\\
Moulay Ismail University of Meknes, Errachidia 52000, Morocco}

\address[BUW]{University of Wuppertal, Chair of Applied and Computational Mathematics,\\
Gaußstrasse 20, 42119 Wuppertal, Germany}


\begin{abstract}
This paper deals with a new epidemiological model of SIRS with stochastic perturbations. 
The primary objective is to establish the existence of a unique non-negative nonlocal solution.
Using the basic reproduction number $\mathscr{R}_0$ derived from the associated deterministic model, we demonstrate the existence of a stationary distribution in the stochastic model. 
In addition, we study the fluctuation of the unique solution of the deterministic problem around the disease-free equilibrium under certain conditions. 
In particular, we reveal scenarios where random effects induce disease extinction, contrary to the persistence predicted by the deterministic model.
The theoretical insights are complemented by numerical simulations, which provide further validation of our findings.
\end{abstract}

\begin{keyword}
epidemiological model \sep stochastic model \sep It{ô} formula \sep Lyapunov function \sep numerical approximations\\
\textit{2020 Mathematics Subject Classification.} 92C60 \sep 91G30 \sep 39A50 \sep 93D05 \sep 33F05
\end{keyword}

\journal{} 


\end{frontmatter}

\section{Introduction}\label{S1}
Stochastic models in epidemiology are mathematical tools that incorporate randomness and uncertainty into the dynamics of infectious diseases. 
Unlike deterministic models, which assume fixed parameters and initial conditions, stochastic models account for the variability and unpredictability of real-world epidemics. 
Stochastic models can capture the effects of random events such as individual contacts, transmission events, recovery times, and environmental fluctuations that can affect the spread and control of disease. 
Moreover, these models are powerful and versatile tools that can help us understand and predict the behavior of infectious diseases in a stochastic world \cite{Jiang2011, Nguyen2020, Lan2019, Tuerxun2021}.

The SIRS model is a compartmental model commonly used to understand and describe the dynamics of infectious diseases. 
The population in this model is divided into three distinct compartments based on their disease status:
\begin{itemize}
\item[$i.$] Susceptible ($\mX$): Individuals in this compartment are susceptible to the infectious agent, but have not yet been infected.

\item[$ii.$] Infectious ($\mY$): Individuals in this compartment are currently infected and can transmit the disease to susceptible individuals.

\item[$iii.$] Recovered ($\mZ$): Individuals who have recovered from infection and developed immunity.
\end{itemize}
Unlike the SIR models (where recovered individuals are assumed to have lifelong immunity), the SIRS models assume that individuals who recover from infection do not acquire permanent immunity. 
Instead, after a period of time, they return to the $\mX$ compartment and become susceptible to the disease again. 
The dynamics of the SIRS model is often described by a system of ordinary differential equations (ODEs), expressed as follows
\begin{equation}\label{E1.1}
\begin{cases}
\;\frac{d \mX(t)}{d t} &= \Lambda  + \eta \mZ(t) - \beta \mX(t) \mY(t) - \mu \mX(t),\\[.1cm]
\;\frac{d \mY(t)}{d t} &= \beta \mX(t) \mY(t) - (\alpha + \mu + \gamma) \mY(t),\\[.1cm]
\;\frac{d \mZ(t)}{d t} &= \gamma \mY(t) - (\eta + \mu) \mZ(t),
\end{cases}
\end{equation}
where
\begin{table}[ht]
\begin{minipage}{0.49\linewidth}
\begin{itemize}
\item[] $\alpha > 0$: Disease-related death rate.

\item[] $\beta > 0$: Effective contact rate.

\item[] $\eta > 0$: Immunity loss rate.
\end{itemize}
\end{minipage}
\hfill
\begin{minipage}{0.49\linewidth}
\begin{itemize}
\item[] $\mu > 0$: Natural death rate.

\item[] $\gamma > 0$: Recovery rate of $\mY$.

\item[] $\Lambda > 0$: Recruitment rate of the population.
\end{itemize}
\end{minipage}
\end{table}\\
To identify equilibria, we set the right side of \eqref{E1.1} to zero. 
This results in the identification of two equilibria in the coordinate space $(S, I, R)$. 
Specifically, the \textit{disease-free equilibrium} (DFE) $E^f(\frac{\Lambda}{\mu}, 0, 0)$ and the \textit{endemic equilibrium} (EE) $E^*(\mX^*, \mY^*, \mZ^*)$, with  
\begin{equation}\label{E1.2}
\mX^* = \frac{\Lambda}{\mu \mathscr{R}_0}, \ \mY^* = \frac{\mu + \eta}{\gamma}\mZ^*, \ \text{ and } \ 
\mZ^* = \frac{\mu\gamma (\alpha + \mu + \gamma)}{\beta \Lambda \bigl(\mu\gamma + (\mu + \eta)(\mu + \alpha)\bigr)}(\mathscr{R}_0 - 1),
\end{equation}
where $\mathscr{R}_0 = \frac{\beta \Lambda}{\mu (\alpha + \mu + \gamma)}$ is the \textit{basic reproduction number} of \eqref{E1.1}. 
More details and information on SIRS models can be found in \cite{Alahakoon2023, Chen2014, Pan2022, Xu2010}.

In contrast, stochastic SIRS models introduce a probabilistic framework that recognizes that infectious disease dynamics are subject to random events and unpredictability in the spread of infectious diseases. 
In the stochastic SIRS model, the transitions between the three compartments are modeled as stochastic processes using stochastic differential equations (SDEs). 
The randomness in the model accounts for variability in individual-level interactions, transmission events, and recovery processes.
The study \cite{Nguyen2020} investigated a stochastic epidemiological model of SIRS characterized by an incidence rate. 
The investigation included the introduction of a real-valued threshold, denoted $\mathscr{R}$, to classify the conditions of extinction and persistence. 
The results showed that if $\mathscr{R} < 0$, the disease is expected to eventually disappear. 
Conversely, when $\mathscr{R} > 0$, the epidemic exhibits strong stochastic permanence.
In line with these results, modifications were applied to the system \eqref{E1.1}, leading to the following expression
\begin{equation}\label{E1.3}
\begin{cases}
\;d \mX(t) &= \bigl( \Lambda  + \eta \mZ(t) - \beta \mX(t) \mY(t) - \mu \mX(t) \bigr)\,dt 
+ \sigma_1 \mX(t) \,d\mB_1(t),\\[.1cm]
\;d \mY(t) &= \bigl( \beta \mX(t) \mY(t) - (\alpha + \mu + \gamma) \mY(t) \bigr)\,dt 
+ \sigma_2 \mY(t) \,d\mB_2(t),\\[.1cm]
\;d \mZ(t) &= \bigl( \gamma \mY(t) - (\eta + \mu) \mZ(t) \bigr)\,dt 
+ \sigma_3 \mZ(t) \,d\mB_3(t).
\end{cases}
\end{equation}
Here, $\mB_1$, $\mB_2$, and $\mB_3$ denote three correlated 
Brownian motions, where $\sigma_1$, $\sigma_2$, and $\sigma_3$ represent the intensities of fluctuations due to the random environment of $\mX$, $\mY$, and $\mZ$, respectively.

The authors in \cite{Lan2019} introduced and studied a stochastic SIRS model with a non-monotone incidence rate under regime switching. 
First, the authors established the existence of a unique positive solution, a prerequisite for the subsequent analysis of the long-term behavior of the stochastic model.
They then developed a threshold dynamics determined by the basic reproduction number $\mathscr{R}_0^s$. 
The results showed that if $\mathscr{R}_0^s < 1$ and under mild additional conditions, the disease could be almost certainly eradicated. 
Conversely, if $\mathscr{R}_0^s > 1$, the density distributions of the solution in $L^1$ could converge to an invariant density, according to the theory of Markov semigroups.
Following these results, the system \eqref{E1.1} underwent a transformation into the following It{ô} SDE
\begin{equation}\label{E1.4}
\begin{cases}
\;d \mX(t) &= \bigl( \Lambda  + \eta \mZ(t) - \beta \mX(t) \mY(t) - \mu \mX(t) \bigr)\,dt 
- \sigma_4 \mX(t) \mY(t) \,d\mB_4(t),\\[.1cm]
\;d \mY(t) &= \bigl( \beta \mX(t) \mY(t) - (\alpha + \mu + \gamma) \mY(t) \bigr)\,dt 
+ \sigma_4 \mX(t) \mY(t) \,d\mB_4(t),\\[.1cm]
\;d \mZ(t) &= \bigl( \gamma \mY(t) - (\eta + \mu) \mZ(t) \bigr)\,dt,
\end{cases}
\end{equation}
where $\sigma_4$ denotes the environmental white noise density and $\mB_4$ is a standard Brownian motion.

Exploration of the stochastic SIRS mathematical epidemiological model provides a compelling and nuanced way to understand infectious disease dynamics. 
By introducing stochastic elements, this modeling approach reflects the inherent uncertainty and indiscriminacy observed in real-world epidemiological scenarios.
Furthermore, this model contributes to our understanding of emerging and re-emerging diseases, where unpredictability plays an important role in the development of mathematical epidemiology, especially in stochastic modeling.
Moreover, researchers not only gain valuable insights into the dynamics of infectious diseases, but also contribute to the development of methodologies essential for public health planning and response in a constantly evolving landscape.\\
Motivated and inspired by the investigations in \cite{Lan2019} and \cite{Nguyen2020}, this study explores the implications of stochastic variation arising from environmental white noise. 
The stochastic counterpart, derived from the deterministic system \eqref{E1.1}, is elucidated in the following section.

The structure of this manuscript is as follows. 
In Section~\ref{S2}, the stochastic SIRS model is introduced and explained. 
Section~\ref{S3} investigates the global existence and positivity of a unique solution. 
Building on this, Section~\ref{S4} establishes the existence of a stationary distribution under certain parametric constraints. 
The dynamics of the solution around the DFE of \eqref{E1.1} is studied in Section~\ref{S5}. 
This investigation leads to the derivation of sufficient conditions for disease extinction in Section~\ref{S6}. 
To illustrate the theoretical results, insightful numerical simulations are presented in Section~\ref{S7}. 
Finally, a comprehensive conclusion of the study is presented in Section~\ref{S8}.

\section{The proposed Model}\label{S2}
Research using the stochastic SIRS model can provide insights into the role of randomness in shaping the course of infectious diseases, the impact of random events on epidemic outcomes, and the effectiveness of interventions in uncertain environments. 
This modeling approach is valuable for understanding the nuanced and probabilistic nature of disease dynamics in real-world populations.

We account for variations in the population environment to study the dynamics of the SIRS model, focusing on its long-term behavior. 
The total population at any time $t$ is denoted by $\mathcal{N}(t)$, and it is categorized into three exclusive compartments, as detailed in the previous section.\\
According to \eqref{E1.1}, we can express the stochastic version of the SIRS model. 
Here, $\sigma_1$, $\sigma_2$, $\sigma_3$, and $\sigma_4$ represent the intensities of the standard Gaussian white noise associated with the independent standard Brownian motion $\mB_1(t)$, $\mB_2(t)$, $\mB_3(t)$, and $\mB_4(t)$, respectively. 
The proposed model then has the following form
\begin{equation}\label{E2.1}
\begin{cases}
\;d \mX(t) &= \bigl( \Lambda  + \eta \mZ(t) - \beta \mX(t) \mY(t) - \mu \mX(t) \bigr)\,dt 
   - \sigma_4 \mX(t) \mY(t) \,d\mB_4(t) 
   + \sigma_1 \mX(t) \,d\mB_1(t),\\[.1cm]
\;d \mY(t) &= \bigl( \beta \mX(t) \mY(t) - (\alpha + \mu + \gamma) \mY(t) \bigr)\,dt 
   + \sigma_4 \mX(t) \mY(t)\, d\mB_4(t) 
   + \sigma_2 \mY(t) \,d\mB_2(t),\\[.1cm]
\;d \mZ(t) &= \bigl( \gamma \mY(t) - (\eta + \mu) \mZ(t) \bigr)\,dt 
   + \sigma_3 \mZ(t)\,d\mB_3(t),
\end{cases}
\end{equation}
with the \textit{initial conditions} (ICs)
\begin{equation}\label{E2.2}
\mX(0) = \mX_0 > 0, \quad \mY(0) = I_0 > 0, \ \text{ and } \ \mZ(0) = R_0 > 0.
\end{equation}
In the following, unless explicitly stated otherwise, we consider a complete probability space denoted by $(\mathfrak{B}, \mathscr{F}, \{\mathscr{F}_t\}_{t\ge 0}, \mathbb{P})$. 
In addition, we use the notation
\begin{equation*}
   \mathbb{W} = \mathbb{R}^*_+ \times \mathbb{R}^*_+ \times \mathbb{R}^*_+ = \bigl\{ (x, y, z) \in \mathbb{R}^3 \ | \ x > 0, \ y > 0, \text{ and } z > 0\bigr\}.
\end{equation*}

\section{Existence, Uniqueness, and Positivity of the Solution}\label{S3}
The main result of this section can be expressed as follows

\begin{theorem}\label{T1}
For any $(\mX_0, \mY_0, \mZ_0) \in \mathbb{W}$, problem \eqref{E2.1} has a unique solution in $\mathbb{W}$ almost surely (a.s.) with unit probability for all $t\ge 0$.
\end{theorem}

\begin{proof}
For arbitrary ICs $(\mX_0, \mY_0, \mZ_0)$, all coefficients in \eqref{E2.1} are continuous and locally Lipschitz. 
Therefore, the system \eqref{E2.1} has a local solution and only one $(\mX(t), \mY(t), \mZ(t))$ for all $t \in [0, \ell_e)$, where $\ell_e$ denotes the blow-up 
time, i.e.\ the time when the solution diverges to infinity.

Next, in order to establish that the solution is global, it is necessary to show that $\ell_e = \infty$ a.s.
To achieve this, let $n_0 > 0$ be sufficiently large so that all ICs lie within $[\frac{1}{n_0}, n_0]$. 
For each integer $n \ge n_0$ we define the \textit{stopping time} as follows
\begin{equation}\label{E3.1}
\ell_n = \inf\Bigl\{ t \in [0, \ell_e) \ | 
\ \min\bigl(\mX(t), \mY(t), \mZ(t)\bigr) \le \frac{1}{n} \ \text{ or } 
\ \max\bigl(\mX(t), \mY(t), \mZ(t)\bigr) \ge n \Bigr\}.
\end{equation}
Throughout this paper, let $\inf(\phi) = 0$, where $\phi$ is the empty set. As $n$ approaches infinity, the sequence $(\ell_n)$ increases according to the definition of $\ell_n$. 
Set $\lim\limits_{n \to \infty} \ell_n = \ell_\infty$, with $\ell_n \ge \ell_\infty$ etc.
If $\ell_\infty = \infty$ a.s., then $\ell_n= \infty$ and $\mX(t) > 0, \ \mY(t) > 0$ and $\mZ(t) > 0$ a.s.\ for $t \ge 0$.

If the above statement is not true, then there exist constants $\mT > 0$ and $\varepsilon \in (0, 1)$ such that
\begin{equation}\label{E3.2}
    \mathbb{P}(\ell_\infty\le \mT) \ge \varepsilon.
\end{equation}
So there exists an integer $n_1 \ge n_0$ such that
\begin{equation*}
   \mathbb{P}(\ell_n \le \mT) \ge \varepsilon, \ \forall n \ge n_1.
\end{equation*}
Let $\mV \colon\mathbb{W}\to\mathbb{R}_+$ be a $\mathcal{C}^2$ function defined by
\begin{equation}\label{E3.3}
\begin{split}
  \mV(t) &:= \mV\bigl(\mX(t), \mY(t), \mZ(t)\bigr)\\ 
  &= \mX(t) + \mY(t) + \mZ(t) - (2 + \delta) 
     - \delta\Bigl(\log\frac{\mX(t)}{\delta} + \frac{1}{\delta}\log\mY(t) + \frac{1}{\delta}\log\mZ(t)\Bigr),
\end{split}
\end{equation}
where $\delta = \frac{\mu + \alpha}{\beta}$. 
From the fact that $x - 1 - \log(\mathrm{x}) \ge 0$ for $x>0$ we have the positivity of the function $\mV$. 
Applying the It{ô} formula to \eqref{E3.3}, we get
\begin{equation*}
\begin{split}
  d\mV(t) &= \ \mLV(t)\,dt + \bigl(\mX(t) - \delta\bigr)\bigl(\sigma_1 \,d\mB_1(t) - \sigma_4 \mY(t) \,d\mB_4(t)\bigr)\\
   &\qquad + \bigl(\mY(t) - 1\bigr)\bigl(\sigma_2 \,d\mB_2(t) + \sigma_4 \mX(t) \,d\mB_4(t)\bigr) 
   + \bigl(\mZ(t) - 1\bigr)\sigma_3 \,d\mB_3(t),
\end{split}
\end{equation*}
with $\mLV\colon\mathbb{W} \to \mathbb{R}_+$ defined by
\begin{equation*}
\begin{split}
\mLV(t) &:= \mLV\bigl(\mX(t), \mY(t), \mZ(t)\bigr)\\ 
  &= \Bigl(1 - \frac{\delta}{\mX(t)}\Bigr)
   \bigl(\Lambda + \eta\mZ(t) - \beta\mX(t)\mY(t) - \mu\mX(t)\bigr) 
   + \frac{\delta\sigma_1^2 + \sigma_4^2 \mY^2(t)}{2}\\
   &\qquad + \Bigl(1 - \frac{1}{\mY(t)}\Bigr)
    \bigl(\beta\mX(t)\mY(t) - (\alpha + \mu + \gamma)\mY(t)\bigr) 
     + \frac{\sigma_2^2 + \sigma_4^2 \mX^2(t)}{2}\\
   &\qquad + \Bigl(1 - \frac{1}{\mZ(t)}\Bigr)
     \bigl(\gamma\mY(t) - (\mu + \eta)\mZ(t)\bigr)
     + \frac{\sigma_3^2}{2} \\
   &= \Lambda + (\delta + 2)\mu + \alpha + \gamma + \eta +\frac{\delta\sigma_1^2 + \sigma_4^2 \mY^2(t) + \sigma_2^2 + \sigma_4^2 \mX^2(t) + \sigma_3^2}{2}\\
   &\qquad - \frac{\delta\Lambda + \delta\eta\mZ(t)}{\mX(t)} 
     - \frac{\gamma\mY(t)}{\mZ(t)} - (\mu + \beta)\mX(t) 
     + \bigl(\beta\delta - (\mu+\alpha)\bigr)\mY(t) - \mu\mZ(t)\\
   &\le \ \Lambda + (\delta + 2)\mu + \alpha + \gamma + \eta +\frac{\delta\sigma_1^2 + \sigma_4^2 \mY^2(t) + \sigma_2^2 + \sigma_4^2 \mX^2(t) + \sigma_3^2}{2} =: \xi.
\end{split}
\end{equation*}
Therefore,
\begin{equation*}
\begin{split}
d\mV(t) &= \xi \,dt + \bigl(\mX(t) - \delta\bigr)
    \bigl(\sigma_1 \,d\mB_1(t) - \sigma_4 \mY(t) \,d\mB_4(t)\bigr)\\
   &\qquad + \bigl(\mY(t) - 1\bigr) \bigl(\sigma_2 \,d\mB_2(t) + \sigma_4 \mX(t) \,d\mB_4(t)\bigr) 
+ \bigl(\mZ(t) - 1\bigr)\sigma_3\, d\mB_3(t).
\end{split}
\end{equation*}
Thus,
\begin{equation}\label{E3.4}
\begin{split}
   \mathbb{E}\bigl(\mV(\ell_n\wedge\mT)\bigr) 
   &\le \mV(0) + \mathbb{E}\biggl(\int_0^{\ell_n\wedge\mT}\xi \,dt\biggr)\\
   &\le \mV(0) + \xi\mT.
\end{split}
\end{equation}
Let $\mathfrak{B}_n = \{\ell_n \le \mT\}$ for $n\ge n_1$. 
According to \eqref{E3.2}, we have $\mathbb{P}(\mathfrak{B}_n) \ge \varepsilon$. 
For $\omega \in \mathfrak{B}_n$ there is at least one of $\mX(\ell_n\wedge\mT)$, $\mY(\ell_n\wedge\mT)$ and $\mZ(\ell_n\wedge\mT)$ such that one of them is equal to $n$ or $\frac{1}{n}$. 
Consequently, $\mV(\ell_n\wedge\mT) := \mV\bigl(\mX(\ell_n, \omega), \mY(\ell_n, \omega), \mZ(\ell_n, \omega)\bigr)$ is not less than $n - 1 - \log(n)$ or $\frac{1}{n} - 1 - \log(n)$. 
Then,
\begin{equation}\label{E3.5}
  \mV(\ell_n\wedge\mT) \ge \Bigl(\frac{1}{n} - 1 - \log(n)\Bigr) \wedge \bigl(n - 1 - \log(n)\bigr).
\end{equation}
Using \eqref{E3.2}, \eqref{E3.4} and \eqref{E3.5}, we get
\begin{equation*}
\begin{split}
\mV(0) + \xi\mT &\ge \mathbb{E}\bigl(1_{\mathfrak{B}_n}(\omega)\mV(\ell_n\wedge\mT)\bigr)\\
&\ge \varepsilon \Bigl(\frac{1}{n} - 1 - \log(n)\Bigr) \wedge 
\bigl(n - 1 - \log(n)\bigr),
\end{split}
\end{equation*}
where $1_{\mathfrak{B}_n}$ is the indicator function of $\mathfrak{B}_n$. Taking the limit as $n$ approaches $\infty$ leads to the following contradiction
\begin{equation*}
   \infty > \mV(0) + \xi\mT = \infty.
\end{equation*}
Then $\ell_\infty = \infty$ a.s.
\end{proof}

\section{The Stationary Distribution}\label{S4}
In this section, we demonstrate the existence of a \textit{stationary distribution} (SD) when the white noise is small. 
Before presenting the main results of this section, we refer to a well-established result of \cite{Khasminskii1980} that will help us in this regard. 
First, we consider the homogeneous Markov process $\mathrm{X}(t)$ in the Euclidean $\ell$-space $\mathscr{E}_\ell$, 
which is governed by the following SDE
\begin{equation}\label{E4.1}
   d\mathrm{X}(t) = \Xi(\mathrm{X})\,dt + \sum_{n=1}^k \sigma_n(\mathrm{X}) \,d\mB_n(t).
\end{equation}
The diffusion matrix is
\begin{equation*}
   \mathcal{A}(\mathrm{x}) = \bigl(a_{i j}(\mathrm{x})\bigr), \quad \text{ with } \quad a_{i j}(\mathrm{x}) = \sum_{n=1}^k \sigma_n^{i}(\mathrm{x})\, \sigma_n^{j}(\mathrm{x}).
\end{equation*}
\begin{assumption} 
There exists a bounded domain $\mU \subset \mathscr{E}_\ell$ with a smooth boundary $\partial\mU$, such that it satisfies the conditions
\begin{itemize}
\item[$(\mH_1)$] Within $\mU$ and its neighborhood, the smallest eigenvalue of $\mathcal{A}$ is bounded away from zero.

\item[$(\mH_2)$] If $\mathrm{x} \in \mathscr{E}_\ell \backslash \mU$, the mean time $\tau$ for a path emerging from $\mathrm{x}$ to reach the set $\mU$ is finite, and $\sup\limits_{\mathrm{x} \in \mathcal{K}} \mathbb{E}^\mathrm{x} \tau<\infty$ holds for every compact subset $\mathcal{K} \subset \mathscr{E}_\ell$.
\end{itemize}
\end{assumption}

\begin{lemma}[{\cite[page 2]{Liu2012}}]\label{L1}
Under the conditions $(\mH_1)$ and $(\mH_2)$, the Markov process $\mathrm{X}(t)$ has a SD $\pi(\cdot)$. 
Let $f$ be an integrable function with respect to $\pi$, then
\begin{equation*}
  \mathbb{P}^\mathrm{x}\biggl(\lim_{\mT\to\infty} \frac{1}{\mT} \int_0^\mT f\bigl(\mathrm{X}(s)\bigr) \,ds
  =\int_{\mathscr{E}_\ell} f(\mathrm{x}) \pi(d \mathrm{x})\biggr) = 1 .
\end{equation*}
\end{lemma}

\begin{remark}\label{R1}
$i.$ To establish the validity of $(\mH_1)$, it suffices to show that $\mV$ is uniformly elliptic in $\mU$, where 
\begin{equation*}
    \mV u = \Xi(\mathrm{x})u_\mathrm{x} + \frac{1}{2}\operatorname{trace}\bigl(\mathcal{A}(\mathrm{x})u_{\mathrm{x}}\bigr).
\end{equation*}
This means that there exists a positive number $\kappa$ such that 
\begin{equation*}
     \sum_{i, j=1}^k a_{ij}(\mathrm{x}) \,\varpi_i \varpi_j \geqslant \kappa |\varpi|^2, 
     \quad \mathrm{x} \in \mU, \ \varpi \in \mathbb{R}^k,
\end{equation*}
(See \cite[page 103]{Gard1988} and \cite[page 349]{Strang1988}).

$ii.$ To check $(\mH_2)$, it suffices to show the existence of a neighborhood $\mU$ and a non-negative $\mathcal{C}^2$-function such that for any $\mathrm{x} \in \mathscr{E}_\ell \backslash \mU$, $\mLV$ is negative (see \cite[page 1163]{Zhu2007}).
\end{remark}

\begin{theorem}\label{T2}
Let $(\mX^{*}, \mY^{*}, \mZ^{*})$ be the EE of \eqref{E1.1}. 
Then, the problem~\eqref{E2.1} has a SD $\pi(\cdot)$ if $\mathscr{R}_0 > 1$ and $0 < \mathscr{C} < \min(\mathscr{D}_1\mX^{* 2}, \mathscr{D}_2\mY^{* 2}, \mathscr{D}_3\mZ^{* 2})$. 
With
\begin{equation}\label{E4.2}
\begin{split}
 \mathscr{D}_1 &= \frac{\mu}{2} - \sigma_1^2 - \frac{(2\mu + \alpha)\mY^*\sigma_4^2}{\beta}, \\
 \mathscr{D}_2 &= \mu+\alpha-\sigma_2^2, \\
 \mathscr{D}_3 &= \frac{\mu(\gamma+2\mu+\alpha) + \eta(2\mu+\alpha)}{\gamma} - \frac{\gamma+2\mu+\alpha}{\gamma}\sigma_3^2, \\
 \mathscr{C} &= \sigma_1^2 \mX^{* 2} + \Bigl(\mY^{* 2}+\frac{2\mu+\alpha}{2\beta} \mY^*\Bigr) \sigma_2^2 + \frac{\gamma+2\mu+\alpha}{\gamma} \mZ^{* 2}\sigma_3^2 
+ \frac{2\mu+\alpha}{\beta} \mX^{* 2}\mY^*\sigma_4^2.
\end{split}
\end{equation}
\end{theorem}

\begin{proof}
The system \eqref{E2.1} can be written as \eqref{E4.1} in the form
\begin{equation*}
\begin{split}
d \begin{pmatrix}
\mX(t)\\ \mY(t)\\ \mZ(t)
\end{pmatrix}
&= \begin{pmatrix}
\Lambda - \beta \mX(t)\mY(t) +\eta \mZ(t) - \mu \mX(t)\\
\beta \mX(t)\mY(t) - (\alpha + \mu + \gamma)\mY(t)\\
\gamma \mY(t) - (\mu + \eta)\mZ(t)
\end{pmatrix} dt 
+ \begin{pmatrix}
\sigma_1 \mX(t)\\ 0\\ 0
\end{pmatrix} d\mB_1(t)\\
&\quad + \begin{pmatrix}
0\\ \sigma_2 \mY(t)\\ 0 \end{pmatrix} d\mB_2(t) 
+ \begin{pmatrix} 0\\ 0\\ \sigma_3 \mZ(t)
\end{pmatrix} d\mB_3(t) 
+ \begin{pmatrix}
-\sigma_4 \mX(t)\mY(t)\\
\sigma_4 \mX(t)\mY(t)\\
0
\end{pmatrix} d\mB_4(t) .
\end{split}
\end{equation*}
In this case, the diffusion matrix is
\begin{equation*}
  \mathcal{A} = \begin{pmatrix}
  \sigma_1^2 \mX^{* 2} + \sigma_4^2 \mX^{* 2}\mY^{* 2} & -\sigma_4^2 \mX^{* 2}\mY^{* 2} & 0\\
  -\sigma_4^2 \mX^{* 2}\mY^{* 2} & \sigma_2^2 \mY^{* 2} + \sigma_4^2 \mX^{* 2}\mY^{* 2} & 0\\
  0 & 0 & \sigma_3^2 \mZ^{* 2}
\end{pmatrix} .
\end{equation*}
Let $\varpi \in \mathbb{R}^3$, then
\begin{equation*}
\begin{split}
   \sum_{i, j=1}^3 a_{i j} \varpi_i \varpi_j 
   &= (\sigma_1^2 \mX^{* 2} + \sigma_4^2 \mX^{* 2} \mY^{* 2}) \varpi_1^2 
    + (\sigma_2^2 \mY^{* 2} + \sigma_4^2 \mX^{* 2} \mY^{* 2}) \varpi_2^2 
    + \sigma_3^2 \mZ^{* 2} \varpi_3^2 - 2 \sigma_4^2 \mX^{* 2} \mY^{* 2} \varpi_1 \varpi_2\\
   &= \sigma_1^2 \mX^{* 2} \varpi_1^2+\sigma_2^2 \mY^{* 2} \varpi_2^2+\sigma_3^2 \mZ^{* 2} \varpi_3^2+\sigma_4^2 \mX^{* 2} \mY^{* 2} (\varpi_1-\varpi_2)^2\\
   &\ge \sigma_1^2 \mX^{* 2} \varpi_1^2+\sigma_2^2 \mY^{* 2} \varpi_2^2+\sigma_3^2 \mZ^{* 2} \varpi_3^2 \ge  \min(\sigma_1^2 \mX^{* 2}, \sigma_2^2 \mY^{* 2}, \sigma_3^2 \mZ^{* 2})|\varpi|^2\\
   &= \kappa |\varpi|^2,
\end{split}
\end{equation*}
which shows that the condition $(\mH_1)$ is satisfied.\\
Since $\mathscr{R}_0 > 1$, then the EE of \eqref{E1.1} is positive, and we have
\begin{equation}\label{E4.3}
    \Lambda = -\eta \mZ^* + \beta \mX^*\mY^* + \mu \mX^*, \quad \beta \mX^*\mY^* 
    = (\alpha + \mu + \gamma)\mY^* \ \text{ and } \ \gamma \mY^* = (\mu + \eta)\mZ^*.
\end{equation}
Recall that
\begin{equation}\label{E4.4}
   (x + y)^2 \le 2x^2 + 2y^2.
\end{equation}
Let
\begin{equation}\label{E4.5}
   \mV(t) = \mV_1(t) + \frac{2\mu + \alpha}{\beta}\mV_2(t) + \frac{2\mu + \alpha}{\gamma}\mV_3(t),
\end{equation}
with
\begin{align*}
 \mV_1(t) &= \frac{1}{2}\bigl(\mX(t) - \mX^* + \mY(t) - \mY^* + \mZ(t) - \mZ^*\bigr)^2,\\
 \mV_2(t) &= \mY^*\Bigl(\frac{\mY(t)}{\mY^*} - 1 - \ln\frac{\mY(t)}{\mY^*}\Bigr),\\
 \mV_3(t) &= \frac{1}{2}\bigl(\mZ(t) - \mZ^*\bigr)^2.
\end{align*}
According to the It{ô} formula, we get
\begin{equation*}
   d\mV(t) = d\mV_1(t) + \frac{2\mu + \alpha}{\beta}d\mV_2(t) + \frac{2\mu + \alpha}{\gamma}d\mV_3(t),
\end{equation*}
where
\begin{equation*}
\begin{split}
  d\mV_1(t) &= \bigl(\mX(t) - \mX^* + \mY(t) - \mY^* + \mZ(t) - \mZ^*\bigr) 
     (d\mX + d\mY + d\mZ) + \frac{1}{2}(d\mX + d\mY + d\mZ)^2 \\
&= \mLV_1(t) \,dt + \bigl(\mX(t) - \mX^* + \mY(t) - \mY^* + \mZ(t) - \mZ^*\bigr)
   \big(\sigma_1 \mX(t) \,d\mB_1(t) + \sigma_2 \mY(t) \,d\mB_2(t)\\
&   \qquad + \sigma_3 \mZ(t) \,d\mB_3(t)\big),\\
d\mV_2(t) &= \Bigl(1-\frac{\mY^*}{\mY(t)}\Bigr) d\mY(t) + \frac{\mY^*}{2 \mY^2(t)}\bigl(d \mY(t)\bigr)^2\\
&= \mLV_2 \,dt + \bigl(\mY(t)-\mY^*\bigr) \bigl(\sigma_2 \,d\mB_2(t) + \sigma_4 \mX(t) \,d\mB_4(t)\bigr),\\
   d\mV_3(t) &= \bigl(\mZ(t)-\mZ^*\bigr) d \mZ(t) + \frac{1}{2}\bigl(d \mZ(t)\bigr)^2\\
&= \mLV_3(t) \,dt + \bigl(\mZ(t)-\mZ^*\bigr) \sigma_3\mZ(t)\, d\mB_3(t),
\end{split}
\end{equation*}
with
\allowdisplaybreaks
\begin{align*}
\mLV_1(t) &= \bigl(\mX(t) - \mX^* + \mY(t) - \mY^* + \mZ(t) - \mZ^*\bigr)
             \bigl[\Lambda - \mu (\mX(t) + \mZ(t)) - (\mu + \alpha) \mY(t)\bigr]\\
&\qquad + \frac{\sigma_1^2 \mX^2(t) + \sigma_2^2 \mY^2(t) + \sigma_3^2 \mZ^2(t)}{2}\\
&\stackrel{\eqref{E4.3}}{\le} -\mu (\mX(t) - \mX^*)^2 - (\mu + \alpha) (\mY(t) - \mY^*)^2 - \mu (\mZ(t) - \mZ^*)^2 - (2\mu+\alpha) (\mX(t) - \mX^*)(\mY(t) - \mY^*)\\
&\qquad + 2\mu (\mX(t) - \mX^*)(\mZ(t) - \mZ^*) - (2\mu+\alpha) (\mY(t) - \mY^*)(\mZ(t) - \mZ^*)\\
&\qquad + \frac{1}{2}\bigl(\sigma_1^2 (\mX(t) - \mX^* + \mX^*)^2 + \sigma_2^2 (\mY(t) - \mY^* + \mY^*)^2 + \sigma_3^2 (\mZ(t) - \mZ^* + \mZ^*)^2\bigr)\\
&\stackrel{\eqref{E4.4}}{\le} -(\mu-\sigma_1^2) (\mX(t) - \mX^*)^2 - (\mu+\alpha-\sigma_2^2) (\mY(t) - \mY^*)^2 - (\mu-\sigma_3^2) (\mZ(t) - \mZ^*)^2\\ 
&\qquad - (2\mu+\alpha) (\mX(t) - \mX^*)(\mY(t) - \mY^*) + 2\mu (\mX(t) - \mX^*)(\mZ(t) - \mZ^*)\\
&\qquad - (2\mu+\alpha) (\mY(t) - \mY^*)(\mZ(t) - \mZ^*) + \sigma_1^2 \mX^{* 2} + \sigma_2^2 \mY^{* 2} + \sigma_3^2 \mZ^{* 2},\\
\mLV_2(t) &= \bigl(\mY(t)-\mY^*\bigr) \bigl[\beta \mX(t)-(\alpha+\mu+\gamma)\bigr]
                +\frac{\mY^*}{2}\bigl(\sigma_2^2+\sigma_4^2 \mX^2(t)\bigr)\\
&\stackrel{\eqref{E4.3}}{\le} \beta (\mX(t) - \mX^*)(\mY(t) - \mY^*) 
      + \frac{\mY^*}{2}\bigl(\sigma_2^2+\sigma_4^2 (\mX(t) - \mX^* + \mX^*)^2\bigr)\\
&\stackrel{\eqref{E4.4}}{\le} \sigma_4^2\mY^*(\mX(t) - \mX^*)^2 + \beta (\mX(t) - \mX^*)(\mY(t) - \mY^*) +
            \frac{\mY^*}{2}\sigma_2^2 + \sigma_4^2 \mY^*\mX^{* 2},\\
\mLV_3(t) &= (\mZ(t)-\mZ^*) \bigl(\gamma \mY(t)-(\mu+\eta) \mZ(t)\bigr) + \frac{1}{2} \sigma_3^2 \mZ^2(t)\\
&\stackrel{\eqref{E4.3}}{\le} \gamma (\mY(t) - \mY^*)(\mZ(t) - \mZ^*) - (\mu+\eta) (\mZ(t) - \mZ^*)^2 + \frac{1}{2} \sigma_3^2 (\mZ(t) - \mZ^* + \mZ^*)^2,\\
&\stackrel{\eqref{E4.4}}{\le} \gamma (\mY(t) - \mY^*)(\mZ(t) - \mZ^*) - (\mu+\eta-\sigma_3^2) (\mZ(t) - \mZ^*)^2 + \sigma_3^2 \mZ^{* 2}.
\end{align*}
Thus,
\begin{equation*}
\begin{split}
\mLV(t) =& \mLV_1(t) + \frac{2\mu + \alpha}{\beta}\mLV_2(t) + \frac{2\mu + \alpha}{\gamma}\mLV_3(t)\\
\le& -\mathscr{D}_1 (\mX(t) - \mX^*)^2 - \mathscr{D}_2 (\mY(t) - \mY^*)^2 - \mathscr{D}_3 (\mZ(t) - \mZ^*)^2 + \mathscr{C},
\end{split}
\end{equation*}
where $\mathscr{D}_1, \mathscr{D}_2, \mathscr{D}_3$ and $\mathscr{C}$ are defined in \eqref{E4.2}.
Note that if $0 < \mathscr{C} < \min(\mathscr{D}_1\mX^{* 2}, \mathscr{D}_2\mY^{* 2}, \mathscr{D}_3\mZ^{* 2})$, then the ellipsoid
\begin{equation*}
   \mathscr{U}\colon \ \mathscr{D}_1 (\mX(t) - \mX^*)^2 + \mathscr{D}_2 (\mY(t) - \mY^*)^2 + \mathscr{D}_3 (\mZ(t) - \mZ^*)^2 = \mathscr{C}
\end{equation*}
is entirely in $\mathbb{W}$. 
Choosing $\mU \subset \mathscr{U}$ so that $\Bar{\mU} \subseteq \mathscr{E}_\ell = \mathbb{W}$, so for $\mathrm{x} \in \mathscr{E}_\ell \backslash \mU$, we have $\mLV(t) \le 0$. 
This implies that $(\mH_2)$ is satisfied. 
Thus, by Lemma~\ref{L1} and remark \ref{R1}, it follows that \eqref{E2.1} has a SD $\pi(\cdot)$.
\end{proof}

\begin{remark}\label{R2}
Under the conditions outlined in Theorem \ref{T2}, the model \eqref{E2.1} has the ergodic property, which means that the positive solution converges to the EE of \eqref{E1.1}.
\end{remark}

\begin{theorem}\label{T3}
Let $\mathcal{N}(t)$ be the total population of \eqref{E2.1}. For $(\mX_0, \mY_0, \mZ_0) \in \mathbb{W}$, the solution of \eqref{E2.1} gives 
\begin{equation}\label{E4.6}
0 < \liminf_{t \to \infty}\mathcal{N}(t) \le \limsup_{t \to \infty}\mathcal{N}(t) < \infty.
\end{equation}
\end{theorem}

\begin{proof}
First, we will show that 
$0 < \liminf\limits_{t \to \infty}\mathcal{N}(t)$. From \eqref{E2.1}, we get
\begin{equation}\label{E4.7}
    d\mathcal{N}(t) = \bigl(\Lambda - \mu \mathcal{N}(t) -\alpha\mY(t)\bigr)\,dt
    + \sigma_1 \mX(t) \,d\mB_1(t) + \sigma_2 \mY(t) \,d\mB_2(t) + \sigma_3 \mZ(t) \,d\mB_3(t).
\end{equation}
Let $\mV(t) = \frac{1}{\mathcal{N}(t)}$. According to the It{ô} formula, we have
\begin{equation*}
\begin{split}
d\mV(t) &= -\frac{1}{\mathcal{N}^2(t)}d\mathcal{N} + \frac{1}{\mathcal{N}^3(t)}(d\mathcal{N})^2\\
        &= \Bigl(\frac{\sigma_1^2 \mX^2(t) + \sigma_2^2 \mY^2(t) + \sigma_3^2 \mZ^2(t)}{\mathcal{N}^3(t)}-\frac{\Lambda - \mu \mathcal{N}(t) -\alpha\mY(t)}{\mathcal{N}^2(t)}\Bigr) \,dt\\
&\qquad - \frac{1}{\mathcal{N}^2(t)}\bigl(\sigma_1 \mX(t)\,d\mB_1(t) + \sigma_2 \mY(t) \,d\mB_2(t) + \sigma_3 \mZ(t) \,d\mB_3(t)\bigr).
\end{split}
\end{equation*}
Since $\frac{\mX(t)}{\mathcal{N}(t)}$, $\frac{\mY(t)}{\mathcal{N}(t)}$, $\frac{\mZ(t)}{\mathcal{N}(t)} \le 1$, then
\begin{equation*}
\begin{split}
d\mV(t) e^t &\le \Bigl(\frac{\sigma_1^2 + \sigma_2^2 + \sigma_3^2}{\mathcal{N}(t)} 
            + \frac{\mu + \alpha}{\mathcal{N}(t)} -\frac{\Lambda}{\mathcal{N}^2(t)} + 1\Bigr)\, e^t\, dt\\
&\qquad - \frac{e^t}{\mathcal{N}^2(t)}\bigl(\sigma_1 \mX(t)\,d\mB_1(t) 
          + \sigma_2 \mY(t)\,d\mB_2(t) 
          + \sigma_3 \mZ(t)\,d\mB_3(t)\bigr).
\end{split}
\end{equation*}
Next, an integration by parts gives
\begin{equation}\label{E4.8}
\begin{split}
\frac{1}{\mathcal{N}(t)} &\le \frac{e^{-t}}{\mathcal{N}_0} + e^{-t}\int_0^t \mathscr{F}\bigl(\mathcal{N}(s)\bigr)\, e^s\, ds\\
& \qquad - e^{-t}\int_0^t \frac{\sigma_1 \mX(s) \,d\mB_1(s) 
   + \sigma_2 \mY(s) \,d\mB_2(s) 
   + \sigma_3 \mZ(s) \,d\mB_3(s)}{\mathcal{N}^2(s)} \,e^s \,ds,
\end{split}
\end{equation}
where $\mathscr{F}(\mathcal{N}(t)) = \frac{\sigma_1^2 + \sigma_2^2 + \sigma_3^2}{\mathcal{N}(t)} + \frac{\mu + \alpha + e^t}{\mathcal{N}(t)} -\frac{\Lambda}{\mathcal{N}^2(t)} + 1$, and for all $\mathrm{X}(t) > 0$, $\mathscr{F}(\mathrm{X}(t)) \le \xi$.

Set $\mathfrak{B} = \bigl\{\omega \, | \, \liminf\limits_{t \to \infty}\mathcal{N}(t, \omega) = 0\bigr\}$.
We need to prove that $\mathbb{P}(\mathfrak{B}) = 0$. 
Assume there exists $\varepsilon \in (0,1)$ such that $\mathbb{P}(\mathfrak{B}) > \varepsilon$. 
Define the stopping time
\begin{equation*}
    \ell_n = \inf\Bigl\{ t \ge 0 \, | \, \mathcal{N}(t, \omega) \le \frac{1}{n}, \: \omega \in \mathfrak{B}\Bigr\}, \quad n \in \mathbb{N}^*.
\end{equation*}
With the definition of $\ell_n$, it is increasing and $\lim\limits_{n \to \infty} \ell_n = \infty$.
By \eqref{E4.8}, we get
\begin{equation}\label{E4.9}
   \mathbb{E}\Bigl(\frac{1}{\mathcal{N}(\ell_n)} 1_{\mathfrak{B}}\Bigr)
   = \mathbb{E}\Bigl(\frac{e^{-\ell_n}}{\mathcal{N}_0} + e^{-\ell_n}\int_0^{\ell_n}  \mathscr{F}\bigl(\mathcal{N}(s)\bigr) \,e^s \,ds\Bigr)
   \le \frac{1}{\mathcal{N}_0} + \xi.
\end{equation}
However,
\begin{equation*}
   \mathbb{E}\Bigl(\frac{1}{\mathcal{N}(\ell_n)} 1_{\mathfrak{B}}\Bigr) 
    \ge n\mathbb{E}(1_{\mathfrak{B}}) \le n\varepsilon \to \infty \quad \text{ as } \quad n \to \infty.
\end{equation*}
This contradicts \eqref{E4.9}. 
Thus, $\mathbb{P}(\mathfrak{B}) = 0$. 
That is to say $0 < \liminf\limits_{t \to \infty}\mathcal{N}(t)$ a.s.

We now concentrate on showing that $\limsup\limits_{t \to \infty}\mathcal{N}(t) < \infty$. 
From \eqref{E4.7} we have
\begin{equation*}
\begin{split}
\mathcal{N}(t) &= \mathcal{N}_0 e^{-\mu t} + e^{-\mu t} \int_0^t\bigl(\Lambda -\alpha\mY(s)\bigr)\, e^{\mu s}\, ds\\
&\qquad + e^{-\mu t}\int_0^t \bigl(\sigma_1 \mX(s) \,d\mB_1(s) + \sigma_2 \mY(s) \,d\mB_2(s) + \sigma_3 \mZ(s) \, d\mB_3(s)\bigr)\, e^{\mu s}\, ds.
\end{split}
\end{equation*}
Using the same method, we find that $\mathbb{P}(\tilde{\mathfrak{B}}) = 0$, 
where $\tilde{\mathfrak{B}} = \bigl\{\omega \ | \ \limsup\limits_{t \to \infty}\mathcal{N}(t, \omega) = \infty\bigr\}$. 
Finally, $\limsup\limits_{t \to \infty}\mathcal{N}(t) < \infty$ a.s.
\end{proof}

\section{Asymptotic behavior around $E^f$}\label{S5}
In this section, we want to illustrate how the solution spirals around the DFE of \eqref{E1.1}. 
The central result of this section is outlined below.

\begin{theorem}\label{T4}
Let $(\mX(t), \mY(t), \mZ(t))$ be the solution of \eqref{E1.1} with ICs in $\mathbb{W}$. 
If $\mathscr{R}_0 < 1$, $\sigma_1^2 < \frac{\mu}{2}$, $\sigma_2^2 < 2(\mu+\alpha)$ and $\sigma_3^2 < \frac{(2\mu+\alpha)(\mu+\eta) + \gamma\mu}{\gamma}$ then
\begin{equation*}
   \limsup_{t \to \infty} \frac{1}{t} \mathbb{E}\int_0^t \biggl(
   \Bigl(\mX(s) - \frac{\Lambda}{\mu}\Bigr)^2 + \mY^2(s) + \mZ^2(s) \biggr) ds \le \frac{\sigma_1^2}{\mathscr{C}}\frac{\Lambda^2}{\mu^2},
\end{equation*}
where $\mathscr{C} = \min\Bigl(\frac{\mu}{2}-\sigma_1^2, \mu+\alpha-\frac{\sigma_2^2}{2}, \frac{(2\mu+\alpha)(\mu+\eta) + \gamma\mu}{\gamma} - \sigma_3^2\Bigr)$.
\end{theorem}

\begin{proof}
Let
\begin{equation*}
    \mV(t) = \frac{1}{2}\Bigl(\mX(t) + \mY(t) + \mZ(t) - \frac{\Lambda}{\mu}\Bigr)^2 
         + \frac{2\mu + \alpha}{\beta}\mY(t) + \frac{2\mu + \alpha}{2\gamma}\mZ^2(t).
\end{equation*}
Using the same approach outlined in the demonstration of Theorem~\ref{T2}, we obtain the following result
\begin{equation*}
\begin{split}
   d\mV(t) &= \mLV(t) dt + \Bigl(\mX(t) + \mY(t) + \mZ(t) - \frac{\Lambda}{\mu}\bigr) 
         \bigl(\sigma_1 \mX(t) \,d\mB_1(t) + \sigma_2 \mY(t) \,d\mB_2(t) + \sigma_3 \mZ(t)\,d\mB_3(t)\bigr)\\
   &\qquad + \frac{2\mu + \alpha}{\beta} \bigl(\sigma_2\mY(t)\,d\mB_2(t) + \sigma_4\mX(t)\mY(t)\,d\mB_4(t)\bigr) + \frac{2\mu + \alpha}{\gamma} \sigma_3\mZ^2(t) \,d\mB_3(t),
\end{split}
\end{equation*}
where
\begin{equation*}
\begin{split}
   \mLV(t) &\le \sigma_1^2\Bigl(\frac{\Lambda}{\mu}\Bigr)^2 - \Bigl(\frac{\mu}{2}-\sigma_1^2\Bigr) 
   \Bigl(\mX(t) - \frac{\Lambda}{\mu}\Bigr)^2 - \Bigl(\mu+\alpha-\frac{\sigma_2^2}{2}\Bigr) \mY^2(t)\\
   &\qquad - \Bigl(\frac{(2\mu+\alpha)(\mu+\eta) + \gamma\mu}{\gamma} - \sigma_3^2\Bigr) \mZ^2(t).
\end{split}
\end{equation*}
Thus,
\begin{equation*}
    \mathbb{E}\mV(t) \le \mV(0) + \sigma_1^2 \Bigl(\frac{\Lambda}{\mu}\Bigr)^2t 
      - \mathscr{C}\mathbb{E}\int_0^t\biggl(\Bigl(\mX(s) - \frac{\Lambda}{\mu}\Bigr)^2 + \mY^2(s) + \mZ^2(s)\biggr) \,ds,
\end{equation*}
$\mathscr{C} = \min\Bigl(\frac{\mu}{2}-\sigma_1^2, \mu+\alpha-\frac{\sigma_2^2}{2}, \frac{(2\mu+\alpha)(\mu+\eta) + \gamma\mu}{\gamma} - \sigma_3^2\Bigr)$. Afterwards, we obtain
\begin{equation*}
  \limsup_{t \to \infty} \frac{1}{t} \mathbb{E}\int_0^t \biggl(
   \Bigl(\mX(s) - \frac{\Lambda}{\mu}\Bigr)^2 + \mY^2(s) + \mZ^2(s) \biggr) \,ds \le \frac{\sigma_1^2}{\mathscr{C}}\frac{\Lambda^2}{\mu^2}.
\end{equation*}
Finally, the proof is complete.
\end{proof}

\begin{remark}\label{R3}
Suppose that $\mathscr{R}_0 < 1$ and $\sigma_1 = 0$. 
In this case, the DFE of \eqref{E2.1} is stochastically asymptotically stable if $\sigma_2^2 < 2(\mu+\alpha)$ and $\sigma_3^2 < \frac{(2\mu+\alpha)(\mu+\eta) + \gamma\mu}{\gamma}$.
\end{remark}

\section{Extinction of the Disease}\label{S6}
In deterministic models, the basic reproduction number $\mathscr{R}_0$ determines the persistence or extinction of the disease. If $\mathscr{R}_0 < 1$, the disease is eliminated, whereas if $\mathscr{R}_0 > 1$, the disease persists in the population. 
However, in this section we will show that if the noise is sufficiently large, the disease will become extinct for the stochastic problem \eqref{E2.1}, even though it may persist for its deterministic version \eqref{E1.1}.

\begin{lemma}[See {\cite[Page 5]{Mao1992}}]\label{L2}
Let $\bigl(\mathcal{M}(t)\bigr)_{t \ge 0}$ be a continuous local martingale with the value $\mathcal{M}(0) = \mathcal{M}_0 = 0$.
Let $\mathscr{G} > 1$, $(\ell_m)$ and $(\vartheta_m)$ are two sequences of $\mathbb{R}_+$ with $\ell_m \to \infty$. Then, for almost all $\omega \in \mathfrak{B}$, there exists a random integer
$m_0(\omega)$ such that, for all $m \ge m_0$,
\begin{equation*}
   \mathcal{M}(t) \le \frac{\vartheta_m}{2} \bigl\langle\mathcal{M}(t), \mathcal{M}(t)\bigr\rangle + \mathscr{G}\frac{\ln (m)}{\vartheta_m}, 
   \quad 0 \le t \le \ell_m.
\end{equation*}
\end{lemma}

Theorem~\ref{T5} provides a criterion for disease eradication based on the interplay between noise intensities and system parameters.

\begin{theorem}\label{T5}
Let $\bigl(\mX(t), \mY(t), \mZ(t)\bigr)$ be the solution of \eqref{E2.1} with $(\mX_0, \mY_0, \mZ_0) \in \mathbb{W}$. Then,
\begin{equation*}
     \limsup_{t \to \infty} \frac{\ln\mY(t)}{t} < \frac{\beta^2}{2\sigma_4^2} - (\alpha + \mu + \gamma) - \frac{\sigma_2^2}{2} \ \ a.s.
\end{equation*}
If $(\alpha + \mu + \gamma) + \frac{\sigma_2^2}{2} > \frac{\beta^2}{2\sigma_4^2}$, then $\mY(t)$ 
 will exponentially go to zero with probability one.
\end{theorem}

\begin{proof}
Put $\mV(t) = \ln \mY(t)$. According to the It{ô} formula, we get
\begin{equation*}
\begin{split}
   d\mV(t) &= \frac{1}{\mY(t)} \,d\mY - \frac{1}{2\mY^2(t)} \,(d\mY)^2\\
    &= \Bigl(\beta\mX(t) - (\alpha+\mu+\gamma) - \frac{\sigma_4^2 \mX^2(t)}{2} - \frac{\sigma_2^2}{2}\Bigr)\, dt + \sigma_4 \mX(t) \,d\mB_4(t) + \sigma_2 \,d\mB_2(t).
\end{split}
\end{equation*}
Next,
\begin{equation}\label{E5.1}
    \ln \mY(t) = \ln\mY_0 + \int_0^t \Bigl(\beta\mX(s) - (\alpha+\mu+\gamma) - \frac{\sigma_4^2 \mX^2(s)}{2} - \frac{\sigma_2^2}{2}\Bigr) \,dt + \mathcal{M}(t) + \sigma_2 \,\mB_2(t),
\end{equation}
where $\mathcal{M}(t) = \sigma_4\int_0^t \mX(s) d\mB_4(s)$ is a continuous local martingale with
\begin{equation*}
    \bigl\langle\mathcal{M}(t), \mathcal{M}(t)\bigr\rangle = \sigma_4^2 \int_0^t \mX^2(s)\, ds.
\end{equation*}
Choosing $\mathscr{G} = 2 > 1$, $\vartheta_m = \vartheta > 0$ and $\ell_m = m$, by Lemma \ref{L2} we have
\begin{equation}\label{E5.2}
    \mathcal{M}(t) \le \frac{\vartheta \sigma_4^2}{2} \int_0^t \mX^2(s) ds + \frac{2\ln (m)}{\vartheta}, \quad 0
\le t \le m.
\end{equation}
Using \eqref{E5.1} and \eqref{E5.2}, we get
\begin{equation}\label{E5.3}
   \ln \mY(t) < \ln\mY_0 + \int_0^t \Bigl(\beta\mX(s) - \frac{(1-\vartheta)\sigma_4^2}{2}\mX^2(s) - (\alpha+\mu+\gamma) - \frac{\sigma_2^2}{2}\Bigr) \,dt
   + \frac{2\ln (m)}{\vartheta} + \sigma_2 \,\mB_2(t).
\end{equation}
We have
\begin{equation*}
\begin{split}
- \Bigl(\frac{(1-\vartheta)\sigma_4^2}{2}\mX^2(s) - \beta\mX(s)\Bigr) 
&= - \frac{(1-\vartheta)\sigma_4^2}{2} \Bigl(\mX^2(s) - \frac{2\beta}{(1-\vartheta)\sigma_4^2}\mX(s)\Bigr)\\
&= - \frac{(1-\vartheta)\sigma_4^2}{2}\Bigl(\mX^2(s) - \frac{\beta}{(1-\vartheta)\sigma_4^2}\Bigr)^2 
 + \frac{\beta^2}{2(1-\vartheta)\sigma_4^2}\\
&\le \frac{\beta^2}{2(1-\vartheta)\sigma_4^2}.
\end{split}
\end{equation*}
Then the inequality \eqref{E5.3} becomes
\begin{equation*}
\ln \mY(t) < \ln\mY_0 + \Bigl(\frac{\beta^2}{2(1-\vartheta)\sigma_4^2} - (\alpha+\mu+\gamma) - \frac{\sigma_2^2}{2}\Bigr)t + \frac{2\ln (m)}{\vartheta} + \sigma_2 \mB_2(t).
\end{equation*}
For $m-1 \le t \le m$, we get
\begin{equation*}
    \frac{\ln \mY(t)}{t} < \frac{\ln\mY_0}{t} + \frac{\beta^2}{2(1-\vartheta)\sigma_4^2} - (\alpha+\mu+\gamma) - \frac{\sigma_2^2}{2} + \frac{2\ln (m)}{\vartheta (m-1)} + \sigma_2 \frac{\mB_2(t)}{t}.
\end{equation*}
By the strong law of large numbers (see \cite[page 12]{Mao2007}) we get $\lim\limits_{t \to \infty} \frac{\mB_2(t)}{t} = 0$. So if $m \to \infty$, then $t \to \infty$. Thus,
\begin{equation*}
\limsup_{t \to \infty}\frac{\ln \mY(t)}{t} \le \frac{\beta^2}{2(1-\vartheta)\sigma_4^2} - (\alpha+\mu+\gamma) - \frac{\sigma_2^2}{2} < \frac{\beta^2}{2\sigma_4^2} - (\alpha+\mu+\gamma) - \frac{\sigma_2^2}{2}.
\end{equation*}
\end{proof}

When $\mathscr{R}_0 > 1$, the solution of \eqref{E2.1} tends to the EE of \eqref{E1.1}, as discussed in Remark~\ref{R2}. 
We have the following corollary

\begin{corollary}\label{C1}
The disease will die out exponentially regardless of the value of $\mathscr{R}_0$ as long as $\sigma_2$ and $\sigma_4$  are sufficiently large ensuring that $(\alpha + \mu + \gamma) + \frac{\sigma_2^2}{2} > \frac{\beta^2}{2\sigma_4^2}$. 
This means that large noises can lead to disease extinction.
\end{corollary}

\section{Numerical Simulations}\label{S7}
In this section, we present numerical simulations in Python to illustrate how noise affects the dynamics of the proposed SIRS model. 
We use the Milstein method \cite{Higham2001}
to perform these simulations. 
Consequently, for $t = 0, \Delta t, 2\Delta t, \dots, n\Delta t$, the SDE model \eqref{E2.1} can be discretized as follows
\begin{equation}\label{E7.1}
\begin{split}
\mX_{k+1} &= \mX_k + (\Lambda-\beta \mX_k \mY_k + \eta\mZ_k - \mu \mX_k) \Delta t
      +\mX_k\Bigl[\sigma_1 \xi_{1, k} \sqrt{\Delta t}+\frac{1}{2} \sigma_1^2(\xi_{1, k}^2-1) \Delta t\Bigr]\\
&\qquad +\mX_k \mY_k \Bigl[\sigma_4 \xi_{4, k} \sqrt{\Delta t}+\frac{1}{2} \sigma_4^2(\xi_{4, k}^2-1) \Delta t\Bigr], \\
\mY_{k+1} &=  \mY_k + \bigl[\beta \mX_k \mY_k-(\gamma+\mu+\alpha) \mY_k\bigr] \Delta t
     +\mY_k \Bigl[\sigma_2 \xi_{2, k} \sqrt{\Delta t}+\frac{1}{2} \sigma_2^2(\xi_{2, k}^2-1) \Delta t\Bigr]\\
&\qquad + \mX_k \mY_k \Bigl[\sigma_4 \xi_{4, k} \sqrt{\Delta t}+\frac{1}{2} \sigma_4^2(\xi_{4, k}^2-1) \Delta t\Bigr], \\
\mZ_{k+1} &= \mZ_k + \bigl(\gamma \mY_k-(\mu+\eta) \mZ_k\bigr) \Delta t
     +\mZ_k \Bigl[\sigma_3 \xi_{3, k} \sqrt{\Delta t}+\frac{1}{2} \sigma_3^2(\xi_{3, k}^2-1) \Delta t\Bigr].
\end{split}
\end{equation}
Here, $\Delta t$ represents the time increment, $\xi_{1, k}, \xi_{2, k}, \xi_{3, k}$, and $\xi_{4, k}$ are $N(0, 1)$-distributed independent random variables with $k = 1, 2, \dots, n$. 
Based on \cite{Anderson1979, Cai2015, Lan2019, Zinihi2024}, the Table~\ref{Tab1}  presents the values of all parameters used in this section.

\setlength{\tabcolsep}{0.5cm}
\begin{table}[hbtp]
\centering
\caption{Parameter values and initial conditions in numerical simulations for \eqref{E2.1}.}\label{Tab1}
\begin{tabular}{|c|c|c|}
\hline
Symbol & Description & Value\\
\hline 
\hline $\alpha$ & Disease-induced death rate & 0.006\\
\hline $\beta$ & Effective contact rate & 0.013\\
\hline $\eta$ & Immunity loss rate & 0.023\\
\hline $\mu$ & Natural death rate & 0.05 (For $\mathscr{R}_0 < 1$)\\
& & 0.006 (For $\mathscr{R}_0 > 1$)\\
\hline $\gamma$ & Recovery rate of $\mY$ & 0.04\\
\hline $\Lambda$ & Recruitment rate of population & 0.33\\
\hline $\mX_{0}$ & Initial susceptible individuals & 10\\
\hline $\mY_{0}$ & Initial infected individuals & 5\\
\hline $\mZ_{0}$ & Initial removed individuals & 2\\
\hline $T$ & Final time & 400\\
\hline
\end{tabular}
\end{table}

In Figures \ref{F1} and \ref{F2}, when all parameters $\sigma_1 = \sigma_2 = \sigma_3 = \sigma_4 = 0$, systems \eqref{E1.1} and \eqref{E2.1} are the same. Under these conditions, Figure \ref{F1} illustrates that when $\mathscr{R}_0 < 1$, the DFE is globally asymptotically stable. Conversely, Figure \ref{F2} shows that when $\mathscr{R}_0 > 1$, the EE becomes globally asymptotically stable.

\begin{figure}[hbtp]
\centering
\includegraphics[height=7.2cm]{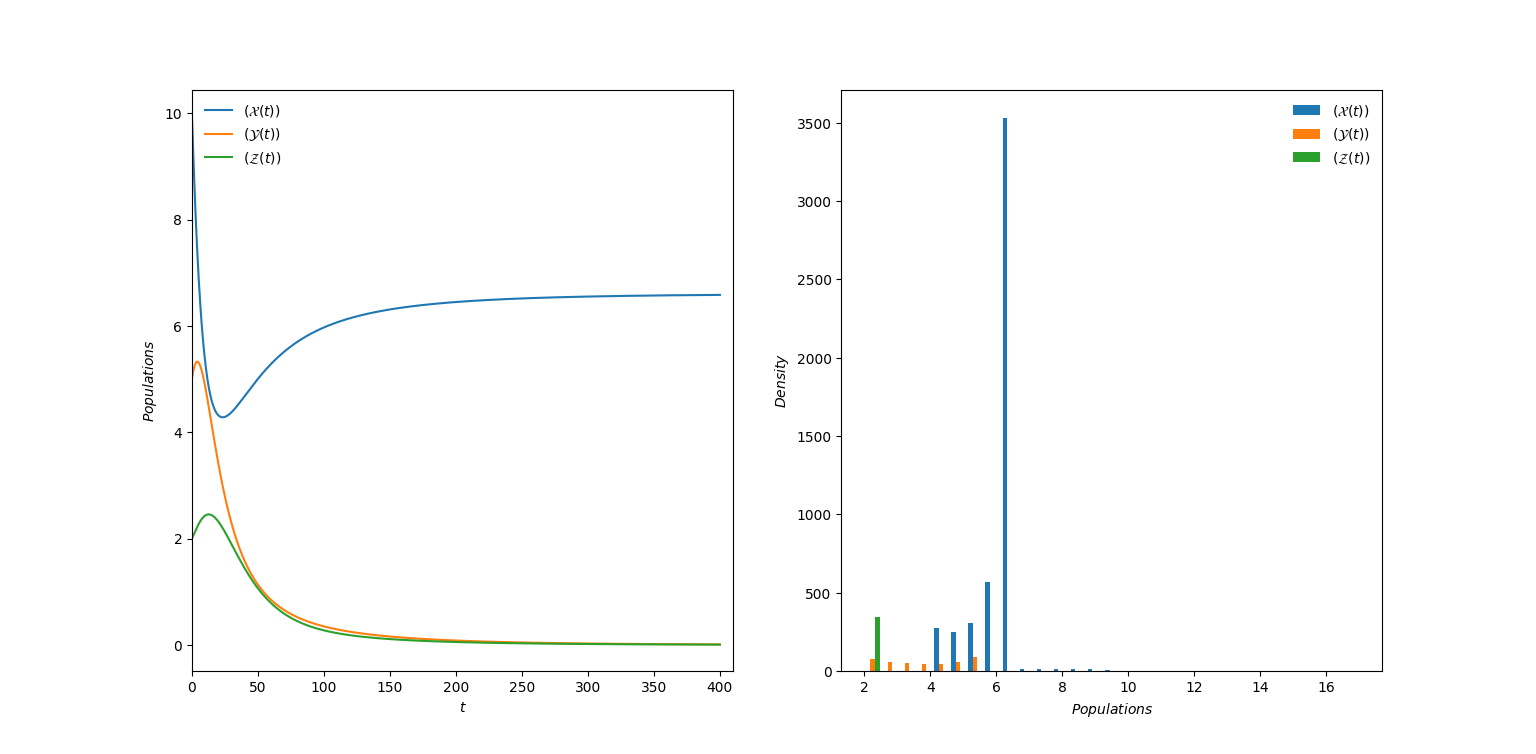}
\caption{Transmission dynamics of the disease for $\mathscr{R}_0 < 1$, with $\sigma_1 = \sigma_2 = \sigma_3 = \sigma_4 = 0$.}\label{F1}
\end{figure}

\begin{figure}[hbtp]
\centering
\includegraphics[height=7.2cm]{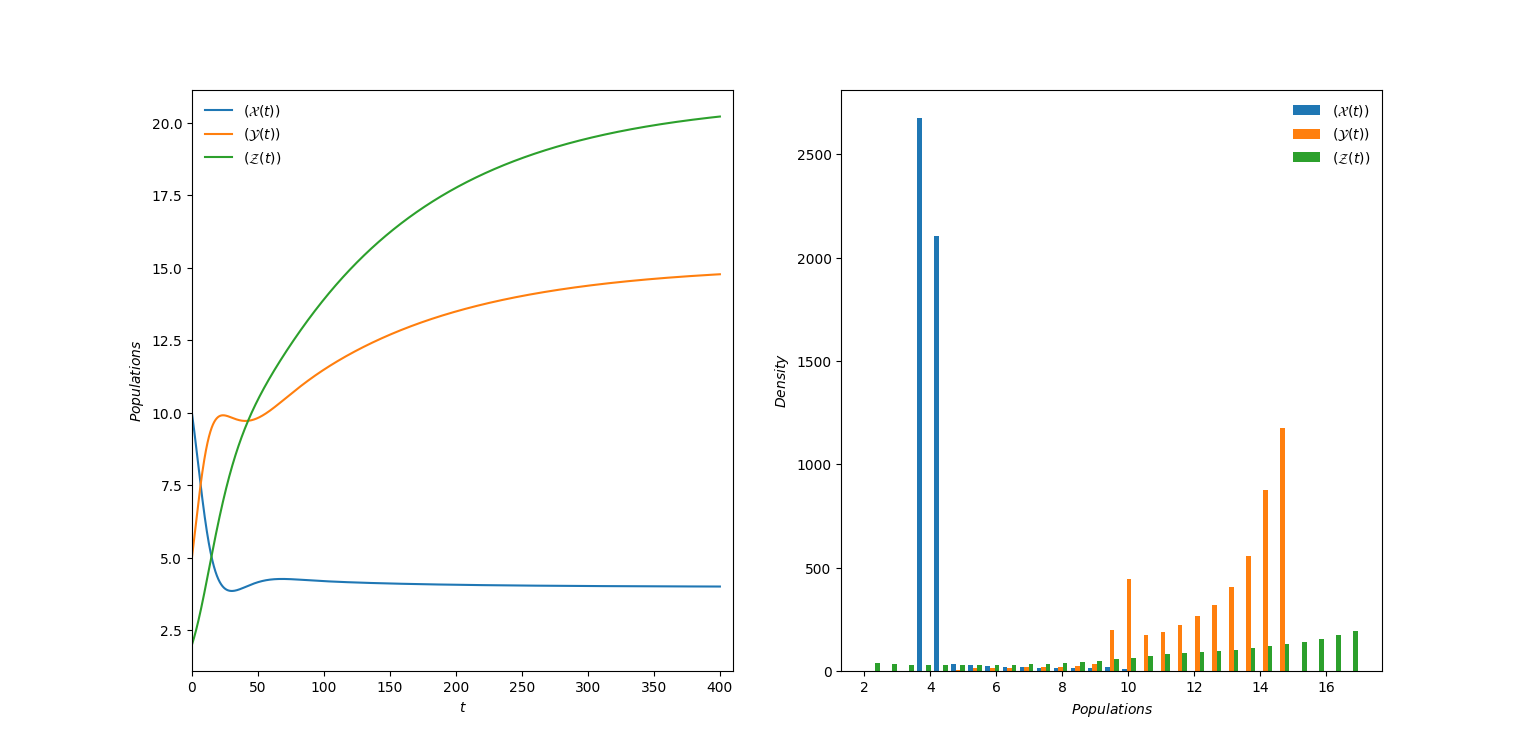}
\caption{Transmission dynamics of the disease for $\mathscr{R}_0 > 1$, with $\sigma_1 = \sigma_2 = \sigma_3 = \sigma_4 = 0$.}\label{F2}
\end{figure}

Figures \ref{F3}--\ref{F8} support the theoretical results described in the previous sections, in particular Theorem~\ref{T2}, Theorem~\ref{T3}, Remark~\ref{R3}, Theorem~\ref{T5}, and Corollary~\ref{C1}, about scenarios where $\mathscr{R}_0$ is both less than and greater than 1, with different values of $\sigma_1$, $\sigma_2$, $\sigma_3$ and $\sigma_4$. 
It is worth noting that the introduction of randomness in shaping the course of infectious diseases has significant implications for both the theoretical and practical aspects of this study. 
The magnitudes of $\sigma_1$, $\sigma_2$, $\sigma_3$, $\sigma_4$, along with their corresponding Brownian motions, play a pivotal role in either propagating the epidemic or leading to its extinction.

\begin{figure}[hbtp]
\centering
\includegraphics[height=8cm]{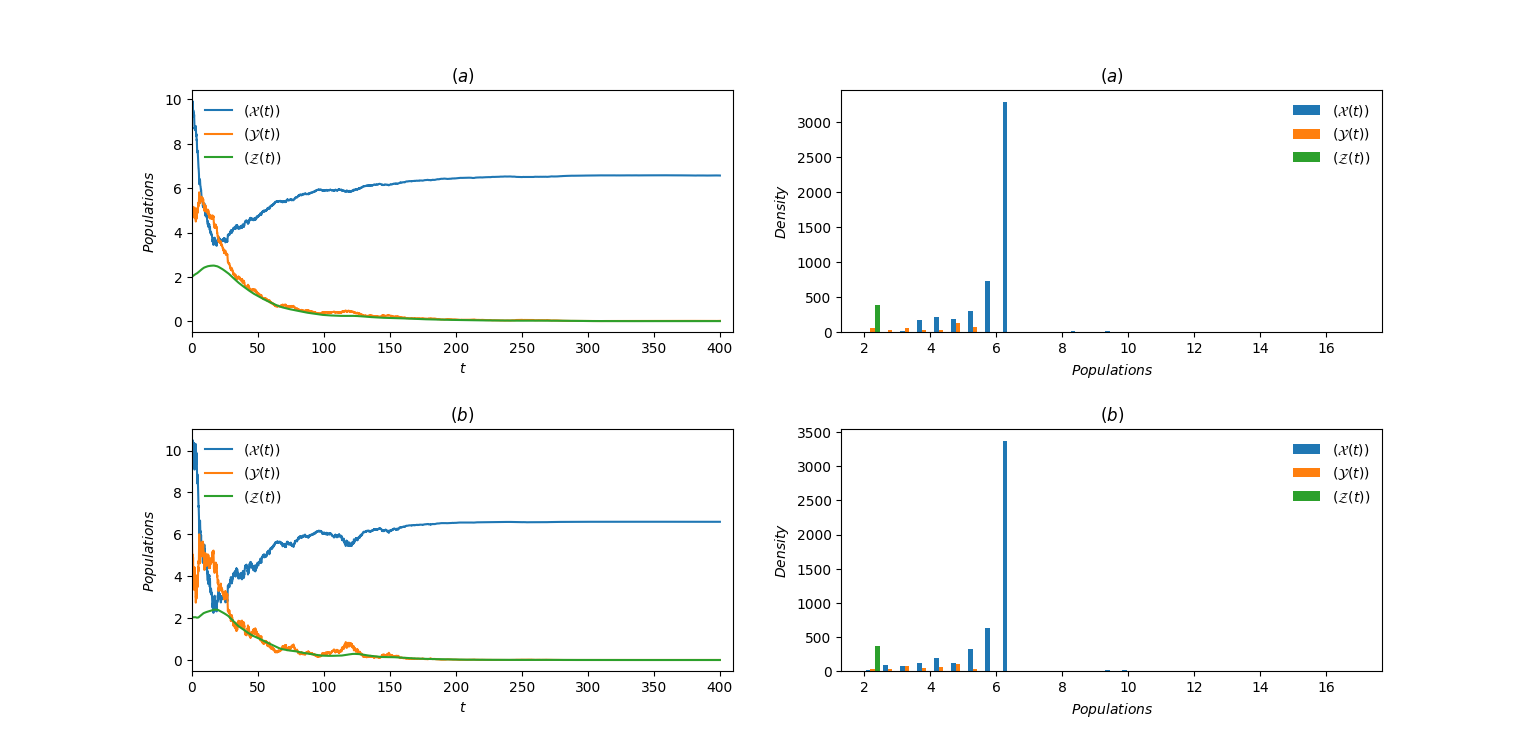}
\caption{Transmission dynamics of the disease for $\mathscr{R}_0 < 1$. Group ($a$): $\sigma_1 = \sigma_2 = \sigma_3 = 0$ and $\sigma_4 = 0.01$. Group ($b$): $\sigma_1 = \sigma_2 = \sigma_3 = 0$ and $\sigma_4 = 0.03$.}\label{F3}
\end{figure}

\begin{figure}[hbtp]
\centering
\includegraphics[height=8cm]{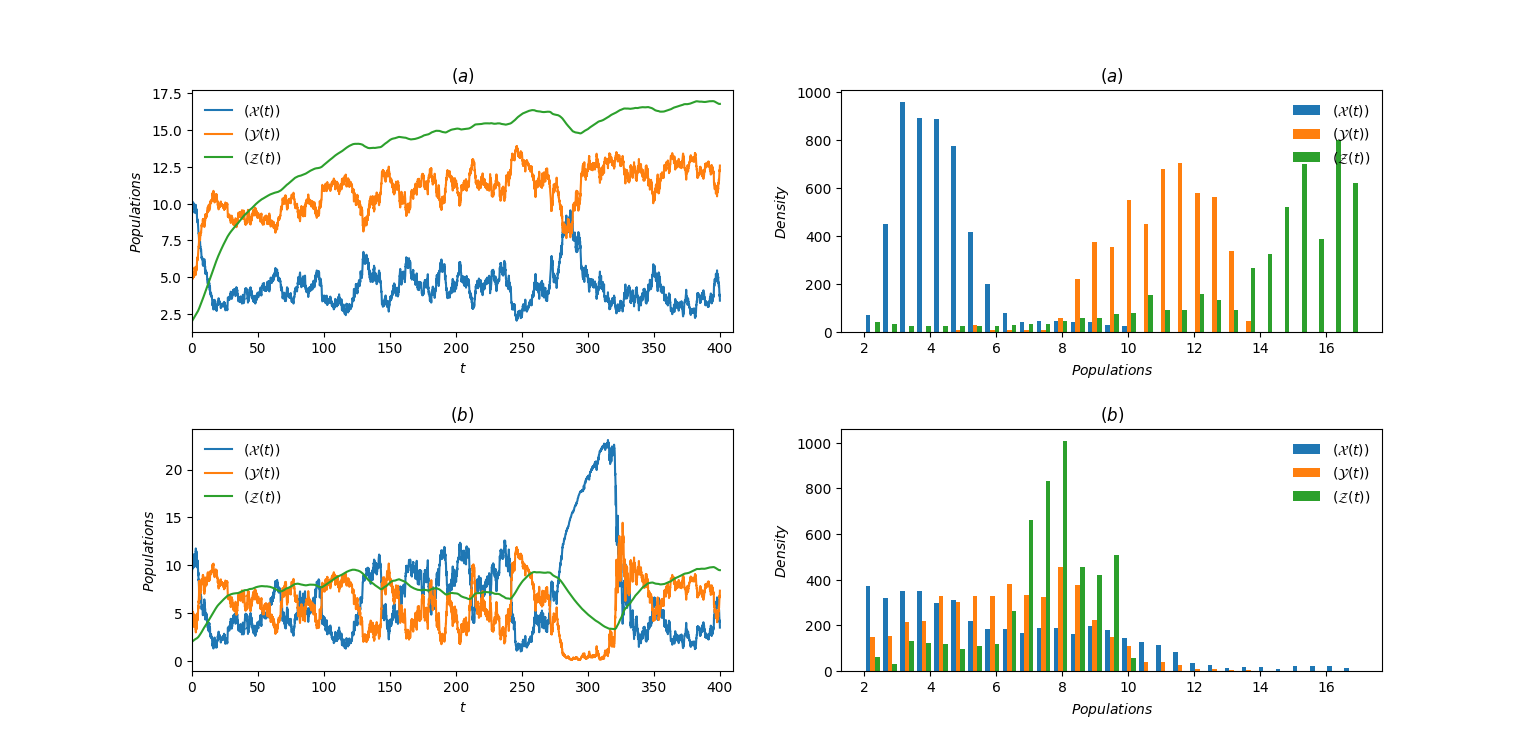}
\caption{Transmission dynamics of the disease for $\mathscr{R}_0 > 1$. Group ($a$): $\sigma_1 = \sigma_2 = \sigma_3 = 0$ and $\sigma_4 = 0.01$. Group ($b$): $\sigma_1 = \sigma_2 = \sigma_3 = 0$ and $\sigma_4 = 0.03$.}\label{F4}
\end{figure}

\begin{figure}[hbtp]
\centering
\includegraphics[height=8cm]{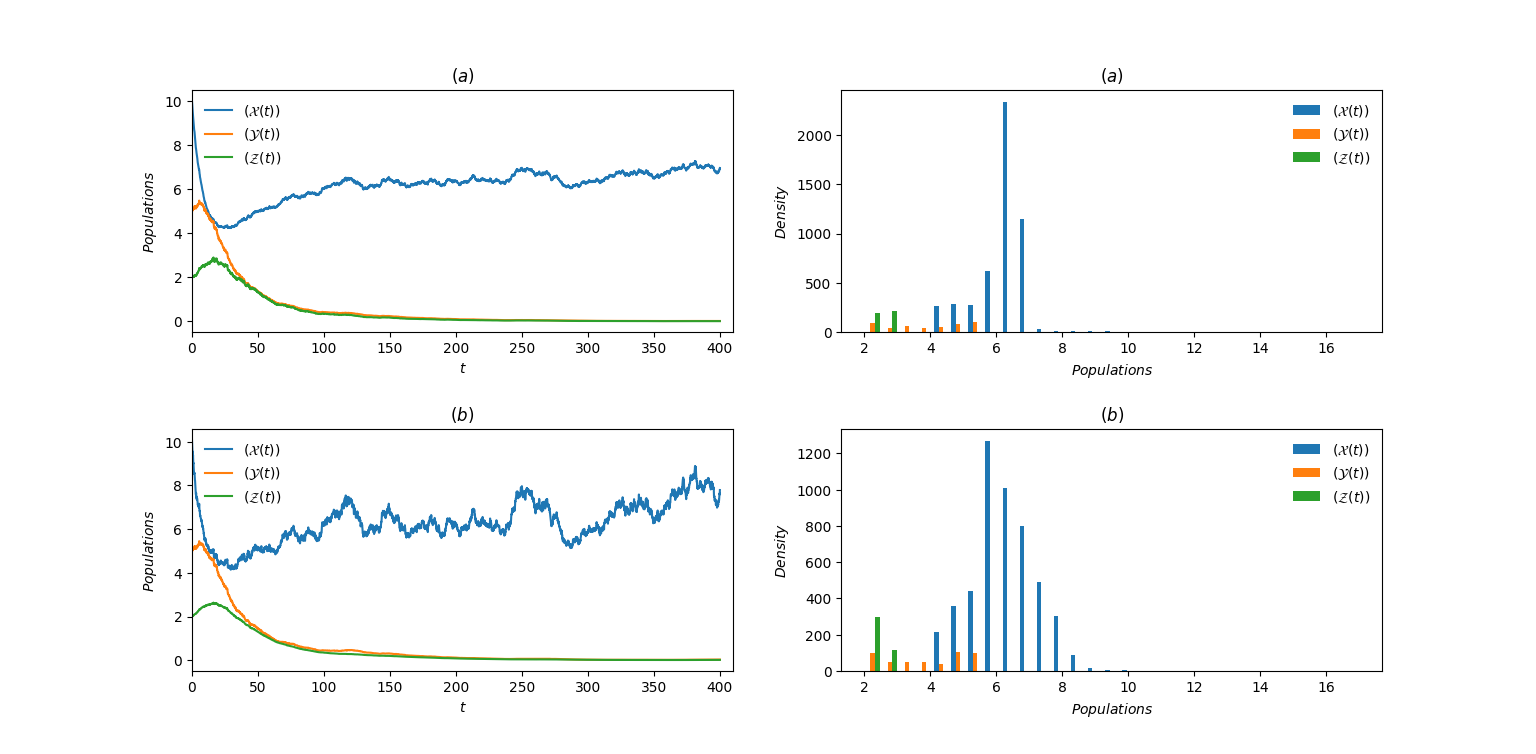}
\caption{Transmission dynamics of the disease for $\mathscr{R}_0 < 1$. Group ($a$): $\sigma_1 = 0.01$, $\sigma_2 = 0.02$, $\sigma_3 = 0.03$ and $\sigma_4 = 0$. Group ($b$): $\sigma_1 = 0.03$, $\sigma_2 = 0.02$, $\sigma_3 = 0.01$ and $\sigma_4 = 0$.}\label{F5}
\end{figure}

\begin{figure}[hbtp]
\centering
\includegraphics[height=8cm]{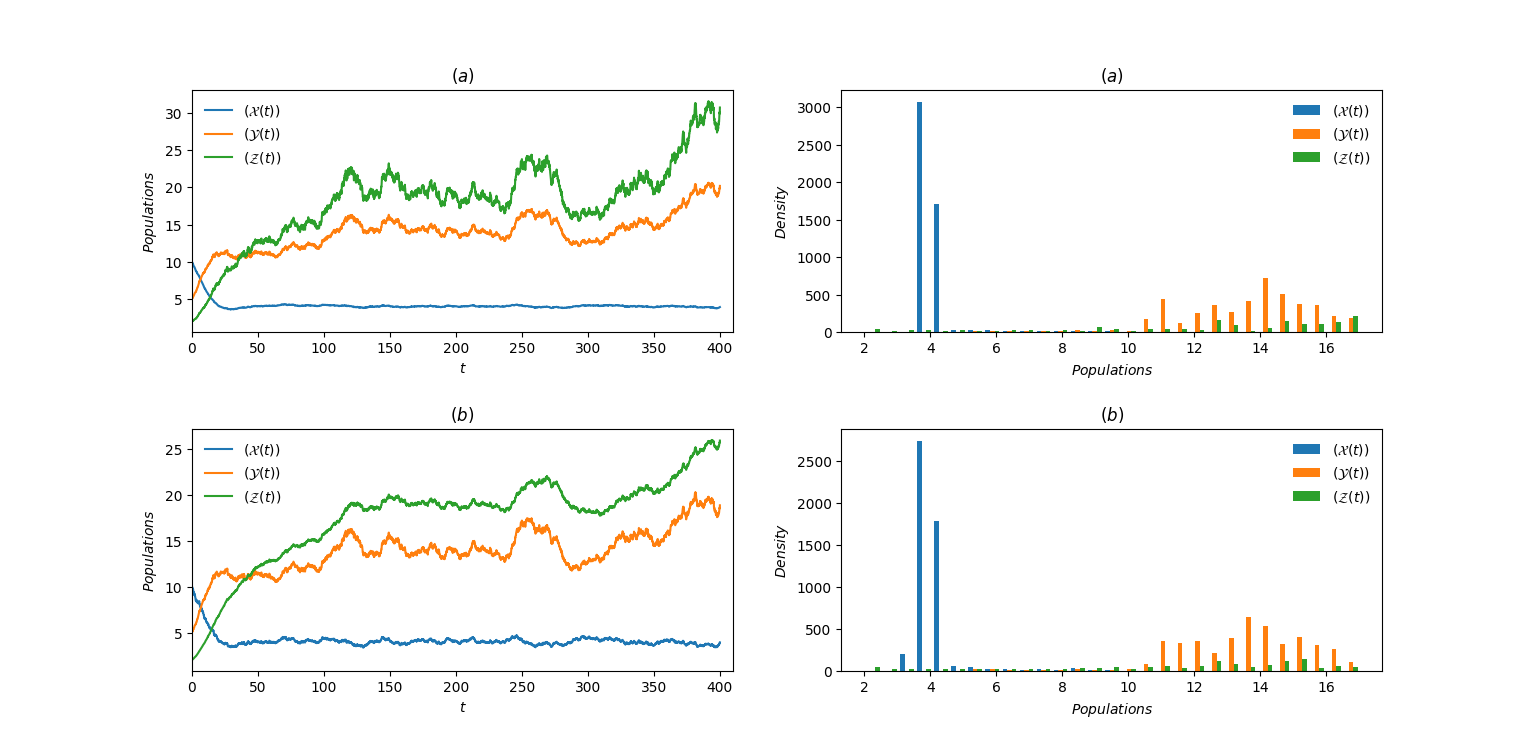}
\caption{Transmission dynamics of the disease for $\mathscr{R}_0 > 1$. Group ($a$): $\sigma_1 = 0.01$, $\sigma_2 = 0.02$, $\sigma_3 = 0.03$ and $\sigma_4 = 0$. Group ($b$): $\sigma_1 = 0.03$, $\sigma_2 = 0.02$, $\sigma_3 = 0.01$ and $\sigma_4 = 0$.}\label{F6}
\end{figure}

\begin{figure}[hbtp]
\centering
\includegraphics[height=8cm]{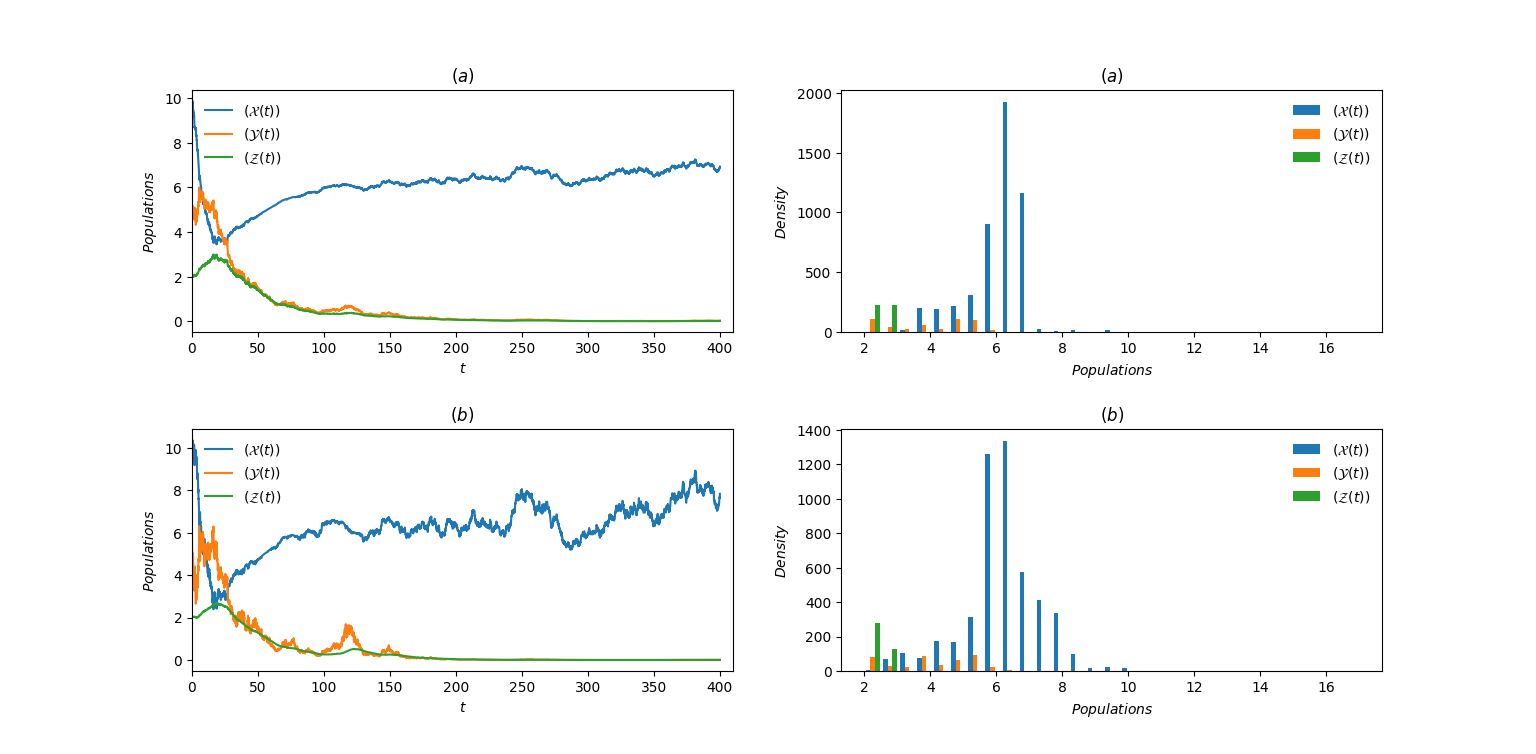}
\caption{Transmission dynamics of the disease for $\mathscr{R}_0 < 1$. Group ($a$): $\sigma_1 = 0.01$, $\sigma_2 = 0.02$, $\sigma_3 = 0.03$ and $\sigma_4 = 0.01$. Group ($b$): $\sigma_1 = 0.03$, $\sigma_2 = 0.02$, $\sigma_3 = 0.01$ and $\sigma_4 = 0.03$.}\label{F7}
\end{figure}

\begin{figure}[hbtp]
\centering
\includegraphics[height=8cm]{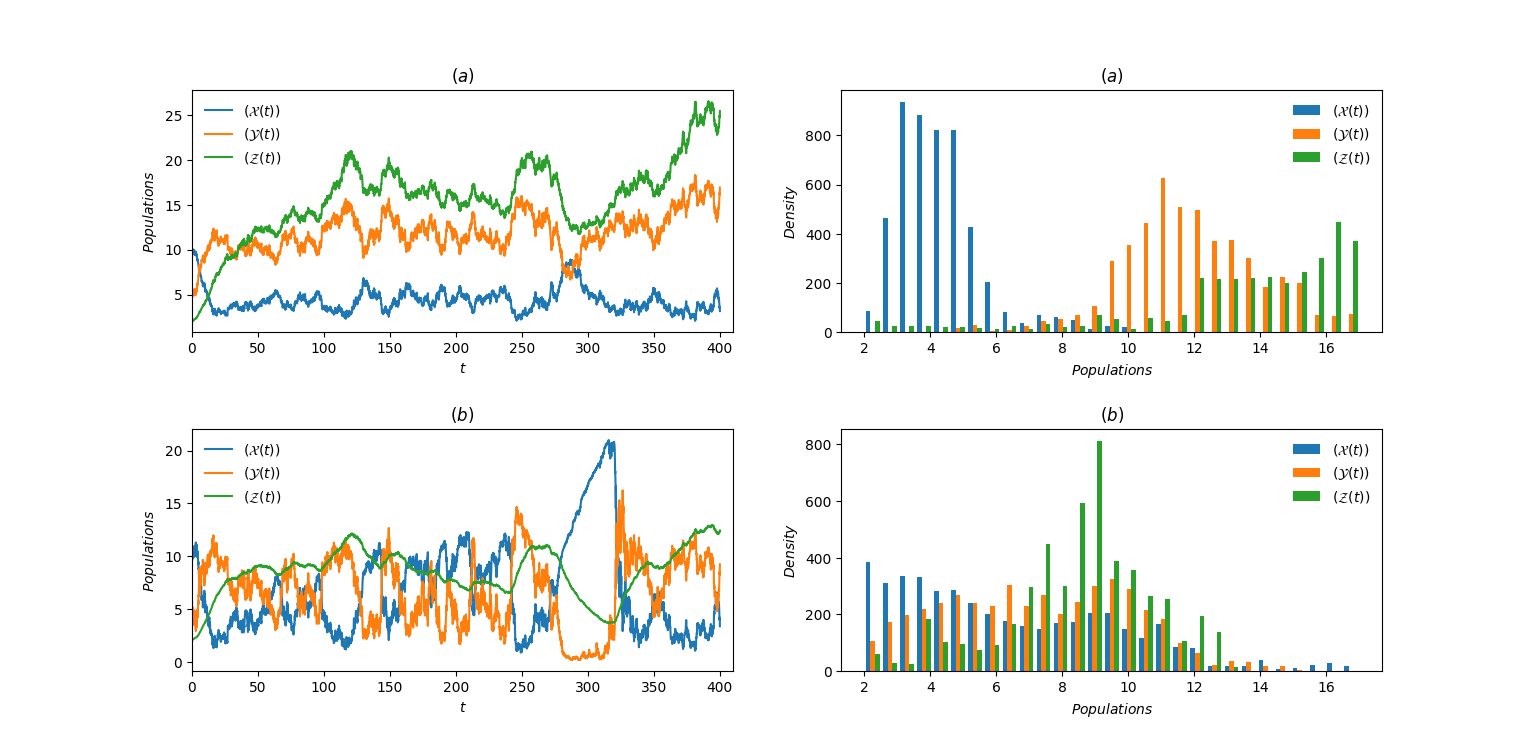}
\caption{Transmission dynamics of the disease for $\mathscr{R}_0 > 1$. Group ($a$): $\sigma_1 = 0.01$, $\sigma_2 = 0.02$, $\sigma_3 = 0.03$ and $\sigma_4 = 0.01$. Group ($b$): $\sigma_1 = 0.03$, $\sigma_2 = 0.02$, $\sigma_3 = 0.01$ and $\sigma_4 = 0.03$.}\label{F8}
\end{figure}

\section{Conclusion}\label{S8}
The SIRS model is a valuable tool for studying the long-term dynamics of infectious diseases in populations where immunity is not durable. Given the profound impact of such diseases on economies and social structures worldwide, accurate modeling is critical. Deterministic equations often fail to capture the complexity of real-world phenomena, especially when stochasticity is involved. Therefore, the use of stochastic models is a more appropriate approach.

In this study, we introduced a stochastic SIRS model tailored to capture the inherent variability in transmission dynamics with changing population environments. Using the theory of stochastic Lyapunov functions, we demonstrated the existence of a unique positive solution. In addition, we explored conditions conducive to disease extinction and analyzed the stationary distribution to identify factors influencing virus extinction. 
Our simulations shed light on the influence of noise intensity on disease transmission and provide valuable insights into the interplay between stochasticity and epidemic dynamics. Through numerical validation using the first-order stochastic Milstein scheme, we confirmed the robustness of our theoretical findings.

In conclusion, our investigation underscores the importance of using stochastic analysis to gain a comprehensive understanding of infectious disease dynamics. By incorporating stochasticity, we enhance our ability to capture the inherent variability and uncertainties of real-world epidemiological systems.


\section*{Declarations}


\subsection*{CRediT author statement} 

\textit{A. Zinihi:} Conceptualization, 
Methodology, 
Software, 
Formal analysis, 
Investigation, 
Writing -- Original Draft, 
Writing -- Review \& Editing, 
Visualization. 

\textit{M. R. Sidi Ammi:} Conceptualization, 
Methodology, 
Validation, 
Formal analysis, 
Investigation,  
Writing -- Original Draft, 
Writing -- Review \& Editing,  
Supervision.

\textit{M. Ehrhardt:} Validation, 
Formal analysis, 
Investigation, 
Writing -- Original Draft, 
Writing -- Review \& Editing, 
Project administration.




\subsection*{Data availability} 
All information analyzed or generated, which would support the results of this work are available in this article.
No data was used for the research described in the article.

\subsection*{Conflict of interest} 
The authors declare that there are no problems or conflicts 
of interest between them that may affect the study in this paper.


\bibliographystyle{acm}
\bibliography{paper}


\end{document}